\documentclass[a4paper,12pt,reqno]{amsart}

\usepackage{amsmath,amssymb,amsthm,tikz,mathrsfs,color}

\newtheorem{lemma}{Lemma}[section]
\newtheorem{theorem}[lemma]{Theorem}
\newtheorem{corollary}[lemma]{Corollary}
\newtheorem{proposition}[lemma]{Proposition}
\newtheorem{conjecture}[lemma]{Conjecture}
\newtheorem{question}[lemma]{Question}

\newtheorem{observation}[lemma]{Observation}

\theoremstyle{definition}
\newtheorem{definition}[lemma]{Definition}
\newtheorem{construction}[lemma]{Construction}

\newtheorem{example}[lemma]{Example}
\newtheorem{remark}[lemma]{Remark}

\DeclareMathOperator{\bl}{_{\!gr.able}}

\newcommand{\iso}{\cong}

\newcommand{\then}{\ensuremath{\Longrightarrow}}
\renewcommand{\iff}{\ensuremath{\Longleftrightarrow}}

\DeclareMathAlphabet{\mathpzc}{OT1}{pzc}{m}{it}

\DeclareMathOperator{\modules}{mod} \renewcommand{\mod}{\modules}

\DeclareMathOperator{\proj}{proj}

\DeclareMathOperator{\add}{add}

\DeclareMathOperator{\End}{End}
\DeclareMathOperator{\Hom}{Hom}
\DeclareMathOperator{\Ext}{Ext}

\DeclareMathOperator{\Aut}{Aut}

\DeclareMathOperator{\Rad}{Rad}

\DeclareMathOperator{\TopOfModule}{top} \renewcommand{\top}{\TopOfModule}

\DeclareMathOperator{\Cok}{Cok}
\DeclareMathOperator{\Img}{Im} \renewcommand{\Im}{\Img}

\DeclareMathOperator{\Cone}{Cone}

\DeclareMathOperator{\gld}{gl.\!dim}

\newcommand{\gr}{_{\rm gr}}

\newenvironment{smallpmatrix}
	       {\left( \! \begin{smallmatrix}}
	       {\end{smallmatrix} \! \right)}

\def\clap#1{\hbox to 0pt{\hss#1\hss}}

\newcommand{\leftsub}[2]{{\vphantom{#2}}_{#1}{#2}}

\usetikzlibrary{arrows,decorations.pathmorphing,decorations.pathreplacing,positioning}

\tikzset{>=stealth',
         vertex/.style={circle,draw=black,inner sep=1.5pt,outer sep=2pt},
         tvertex/.style={inner sep=1pt,font=\scriptsize},
         gap/.style={fill=white,inner sep=1pt}}

\newcommand{\arrow}[2][20]
 {
  \hspace{-5pt}
  \begin{tikzpicture}
   \node (A) at (0,0) {};
   \node (B) at (#1pt,0) {};
   \draw [#2] (A) -- (B);
  \end{tikzpicture}
  \hspace{-5pt}
 }

\newcommand{\arrowl}[3][20]
 {
  \hspace{-5pt}
  \begin{tikzpicture}
   \node (A) at (0,0) {};
   \node (B) at (#1pt,0) {};
   \draw [#2] (A) -- node [above] {$#3$} (B);
  \end{tikzpicture}
  \hspace{-5pt}
 }

\renewcommand{\to}[1][20]{\arrow[#1]{->}}
\newcommand{\tol}[2][20]{\arrowl[#1]{->}{#2}}

\newcommand{\epi}[1][20]{\arrow[#1]{->>}}

\newcommand{\monol}[2][20]{\arrowl[#1]{>->}{#2}}

\renewcommand{\mapsto}[1][20]{\arrow[#1]{|->}}

\title[Image of derived cat.\ in cluster cat.]{The image of the derived category in the cluster category}

\author[Amiot]{Claire Amiot}
\address{Institut de Recherche Math\'ematique Avanc\'ee, 7 rue Ren\'e Descartes, 67084 Strasbourg Cedex, France}
\email{amiot@math.unistra.fr}

\author[Oppermann]{Steffen Oppermann}
\address{Institutt for matematiske fag\\ NTNU\\ 7491 Trondheim\\ Norway}
\email{Steffen.Oppermann@math.ntnu.no}

\begin{document}

\begin{abstract}
Cluster categories of hereditary algebras have been introduced as orbit categories of their derived categories. Keller has pointed out that for non-hereditary algebras orbit categories need not be triangulated, and he introduced the notion of triangulated hull to overcome this problem.

In this paper we study the image if the natural functor from the bounded derived category to the cluster category, that is we investigate how far the orbit category is from being the cluster category.

We show that for wide classes of non-piecewise hereditary algebras the orbit category is never equal to the cluster category.
\end{abstract}

\maketitle

\tableofcontents

\section{Introduction}

In \cite{CC} the first author introduced cluster categories $\mathscr{C}_{\Lambda}$ for finite-dimensional algebras $\Lambda$ with $\gld \Lambda \leqslant 2$. They are defined to be triangulated orbit categories in the sense of Keller (see \cite{K_orbit}). For hereditary algebras $\Lambda$, Keller has shown that the cluster category $\mathscr{C}_{\Lambda}$ and the orbit category (see Definition~\ref{def.clustercat_orbitcat}) coincide. However he pointed out that this need not be the case in general. The aim of this paper is to investigate what this difference is.

Our approach on this aim is as follows: By \cite{CC}, under a technical assumption which we will make throughout this paper (called $\tau_2$-finiteness -- see Definition~\ref{def.tau-2-fin}), the image of the algebra $\Lambda$ is cluster tilting in the cluster category. We denote the endomorphism ring of this cluster tilting object by $\widetilde{\Lambda}$. It follows that there is a natural functor $\mathscr{C}_{\Lambda} \to \mod \widetilde{\Lambda}$. We will point out that $\widetilde{\Lambda}$ carries a natural $\mathbb{Z}$-grading. Then we will show the following:
\begin{theorem}[see Theorem~\ref{theorem.imgr}]
The orbit category coincides with the cluster category if and only if any $\widetilde{\Lambda}$-module is gradable.

More precisely, the objects in the orbit category are precisely those objects in the cluster category, for which the corresponding $\widetilde{\Lambda}$-module is gradable.
\end{theorem}

We then use this theorem to carry over results on gradability of modules to the cluster category setup. In particular we show that objects outside the orbit category come in $1$-parameter families (see Theorem~\ref{theorem.param-ungr}). It follows (see Corollary~\ref{corollary.rigid-image}) that all rigid, and hence all cluster tilting objects in the cluster category lie inside the orbit category.

We then focus on the question when the orbit category coincides with the cluster category. The most ambitious hope one could have in this direction is the following:

\begin{conjecture} \label{conjecture}
The orbit category coincides with the cluster category for an algebra $\Lambda$ if and only if $\Lambda$ is piecewise hereditary.
\end{conjecture}

Note that this conjecture is of a similar flavor as the following question.

\begin{question}[{Skowro{\'n}ski \cite[Question~1]{Skow_conj}}] \label{question.skowronski}
Let $\Lambda$ be a finite dimensional algebra, ${\rm T}(\Lambda) = \Lambda \ltimes D \Lambda$ the trivial extension, which may be considered as a graded algebra by putting $\Lambda$ in degree $0$ and $D \Lambda$ in degree $1$. When is the push-down functor from graded ${\rm T}(\Lambda)$-modules to ungraded ${\rm T}(\Lambda)$-modules dense?
\end{question}

It should be noted that the ``if'' part of Conjecture~\ref{conjecture} holds (see \cite[Theorem in Section~4]{K_orbit}). Here we collect evidence for the ``only if'' part. To this end we show the following two results:

\begin{theorem}[see Theorem~\ref{theorem.fracCY}]
Assume one object in the derived category of $\Lambda$ satisfies a fractional Calabi-Yau type condition with Calabi-Yau dimension $\neq 1$. Then Conjecture~\ref{conjecture} holds.
\end{theorem}

\begin{theorem}[see Theorem~\ref{thm.cycles}]
Assume the quiver of $\Lambda$ contains an oriented cycle. Then the orbit category is strictly smaller than the cluster category. In particular Conjecture~\ref{conjecture} holds.
\end{theorem}

\section{Background and Notation}

We assume all our algebras to be finite-dimensional associative algebras over an algebraically closed base field $k$. Moreover, all categories are $k$-categories.

For an algebra $\Lambda$ we denote by $\mathscr{D}_{\Lambda}$ the bounded derived category of the category of finitely generated left $\Lambda$-modules $\mod \Lambda$.

\begin{definition}
Let $\mathscr{T}$ be a triangulated category. A \emph{Serre functor} of $\mathscr{T}$ is an autoequivalence $\mathbb{S}$, such that there is a functorial isomorphism
\[ \Hom_{\mathscr{T}}(X, Y) \iso D \Hom_{\mathscr{T}}(Y, \mathbb{S} X). \]

We denote by $\mathbb{S}_2 = \mathbb{S}[-2]$ the second desuspension of $\mathbb{S}$.
\end{definition}

Note that if $\Lambda$ is an algebra of finite global dimension, then the $\mathscr{D}_{\Lambda}$ has a Serre functor, which is given by $\mathbb{S} = D \Lambda \otimes_{\Lambda}^L -$.

\begin{definition}
A triangulated category $\mathscr{T}$ is called \emph{$d$-Calabi-Yau} if the $d$-th suspension $[d]$ is a Serre functor on $\mathscr{T}$.
\end{definition}

\begin{definition}
Let $\mathscr{T}$ be a triangulated category. An object $T \in \mathscr{T}$ is called ($2$-)cluster tilting if
\begin{align*}
\add T & = \{ X \in \mathscr{T} \mid \Hom_{\mathscr{T}}(T, X[1]) = 0 \} \\
& = \{ X \in \mathscr{T} \mid \Hom_{\mathscr{T}}(X, T[1]) = 0 \}.
\end{align*}
Note that in case $\mathscr{T}$ is $2$-Calabi-Yau the two subcategories on the right automatically coincide.
\end{definition}

\subsection{Orbit categories and generalized cluster categories} \label{subsect.intro_Amiotcluster}

\begin{definition} \label{def.clustercat_orbitcat}
For an algebra $\Lambda$ of global dimension $\gld \Lambda \leq 2$ we denote by $\mathscr{C}_{\Lambda}$ the \emph{cluster category} of $\Lambda$ as introduced in \cite{CC}.

We denote by $\pi \colon \mathscr{D}_{\Lambda} \to \mathscr{C}_{\Lambda}$ the natural functor from the derived category to the cluster category.

We denote by $\mathscr{D}_{\Lambda} / (\mathbb{S}_2)$ the \emph{orbit category} of $\mathscr{D}_{\Lambda}$ modulo $\mathbb{S}_2$, that is the category with the same objects as $\mathscr{D}_{\Lambda}$, but with morphisms sets
\[ \Hom_{\mathscr{D}_{\Lambda} / \mathbb{S}_2}(X,Y) = \bigoplus_{i \in \mathbb{Z}} \Hom_{\mathscr{D}_{\Lambda}}(\mathbb{S}_2^i X, Y). \]
Then $\pi$ factors into
\[ \mathscr{D}_{\Lambda} \to[40] \mathscr{D}_{\Lambda} / (\mathbb{S}_2) \monol[40]{\substack{\text{fully} \\ \text{faithful}}} \mathscr{C}_{\Lambda}. \]
\end{definition}

\begin{theorem}[Keller \cite{K_orbit}]
Assume $\Lambda$ is hereditary. Then the functor $\mathscr{D}_{\Lambda} / (\mathbb{S}_2) \to \mathscr{C}_{\Lambda}$ of Definition~\ref{def.clustercat_orbitcat} is an equivalence.
\end{theorem}

\begin{theorem}[{\cite[Proposition~4.7]{CC}}] \label{theo.tensoralg}
Let $\Lambda$ be an algebra with $\gld \Lambda \leq 2$. Then
\[ \widetilde{\Lambda} = \End_{\mathscr{C}_{\Lambda}}(\pi \Lambda) = \oplus_{i \geq 0} \Hom_{\mathscr{D}}(\Lambda, \mathbb{S}_2^{-i} \Lambda) = T_{\Lambda} \Ext^2_{\Lambda}( D \Lambda, \Lambda), \]
where $T_{\Lambda}$ denotes the tensor algebra over $\Lambda$.
\end{theorem}

In particular $\widetilde{\Lambda}$ has a natural positive $\mathbb{Z}$-grading.

Note that it follows from the right hand side above that $\widetilde{\Lambda}$ is generated in degrees $0$ and $1$, and that minimal generators in degree $1$ correspond to a minimal set of relations in $\Lambda$.

\begin{definition} \label{def.tau-2-fin}
An algebra $\Lambda$ is called \emph{$\tau_2$-finite} if $\gld \Lambda \leq 2$, and $\widetilde{\Lambda}$ is non-zero in only finitely many degrees.
\end{definition}

\begin{theorem}[{\cite[Theorem~4.10]{CC}}] \label{theorem.is2cy}
Let $\Lambda$ be $\tau_2$-finite. Then $\mathscr{C}_{\Lambda}$ is a 2-CY category, and $\pi \Lambda$ is a cluster tilting object.
\end{theorem}

\section{Connection to gradable modules}

Throughout this section we assume $\Lambda$ to be $\tau_2$-finite. Then by Theorem~\ref{theorem.is2cy} the cluster category is 2-CY. Hence, by the arguments of \cite[Theorem~2.2]{BMR_cta}, for any cluster tilting object $T$ the functor
\begin{align*}
\mathtt{M}_T \colon \mathscr{C}_{\Lambda} & \to[30] \mod \End_{\mathscr{C}}(T) \\
X & \mapsto[30] \Hom_{\mathscr{C}} (T, X)
\end{align*}
induces an equivalence $\mathscr{C}_{\Lambda} / (T[1]) \tol[30]{\approx} \mod \End_{\mathscr{C}}(T)$. In particular we have
\[ \mathtt{M} := \mathtt{M}_{\pi \Lambda} \colon \mathscr{C}_{\Lambda} / (\pi \Lambda [1]) \tol[30]{\approx} \mod \widetilde{\Lambda}. \]

The composition $\mathtt{M} \pi$ is the functor
\begin{align*}
\mathtt{M} \pi \colon \mathscr{D}_{\Lambda} & \to[30] \mod \widetilde{\Lambda} \\
X & \mapsto[30] \bigoplus_{i \in \mathbb{Z}} \Hom_{\mathscr{D}_{\Lambda}}(\Lambda, \mathbb{S}_2^{-i} X).
\end{align*}
We now denote the category of graded $\widetilde{\Lambda}$-modules by $\mod\gr \widetilde{\Lambda}$. By the formula above it is clear that $\mathtt{M} \pi X$ carries the structure of a graded $\widetilde{\Lambda}$-module for any $X \in \mathscr{D}_{\Lambda}$, and hence that $\mathtt{M} \pi$ factors through the forgetful functor $\mod\gr \widetilde{\Lambda} \to[30] \mod \widetilde{\Lambda}$. We denote the (full, isomorphism-closed) subcategory of $\mod \widetilde{\Lambda}$ consisting of modules in the image of this forgetful functor by $\mod\bl \widetilde{\Lambda}$.

\begin{theorem} \label{theorem.imgr}
Let $\Lambda$ be a $\tau_2$-finite algebra. Then for $X \in \mathscr{C}_{\Lambda}$ the following are equivalent:
\begin{enumerate}
\item $X$ is in the image of $\pi$, that is there is $Y \in \mathscr{D}_{\Lambda}$ with $X \iso \pi Y$, and
\item $\mathtt{M} X \in \mod\bl \widetilde{\Lambda}$.
\end{enumerate}
\end{theorem}

\begin{proof}
By the observations above we know that (1) $\then$ (2).

For the implication (2) $\then$ (1) note that both conditions (1) and (2) are satisfied for some $X$ if and only if they are satisfied for all indecomposable direct summands of $X$. Hence we may assume $X$ to be indecomposable.

If $X \in \add \pi \Lambda [1]$ then clearly $X$ is in the image of $\pi$, so the implication holds. If $X \not\in \add \pi \Lambda [1]$ it suffices to show that $\mathtt{M} X \iso \mathtt{M} \pi Y$ for some $Y \in \mathscr{D}_{\Lambda}$ (since $\mathtt{M}$ preserves isomorphism classes of objects without direct summands in $\add \pi \Lambda [1]$). By Property~(2) we know that there is $\widetilde{ \mathtt{M} X} \in \mod\gr \widetilde{\Lambda}$ which is mapped to $\mathtt{M} X$ by the forgetful functor.

Let
\[ P_1 \tol[30]{p} P_0 \epi[30] \widetilde{ \mathtt{M} X} \]
be the beginning of a graded projective resolution of $\widetilde{ \mathtt{M} X}$. Now note that by the Yoneda Lemma the functor
\[ \add \{ \mathbb{S}_2^i \Lambda \mid i \in \mathbb{Z} \} \to[30] \proj\gr \widetilde{\Lambda} \]
induced by $\mathtt{M} \pi$ is an equivalence. Hence $p$ can be lifted to a map $P_1^* \tol[30]{p^*} P_0^*$ in $\mathscr{D}_{\Lambda}$. Now we set $Y = \Cone p^*$. Applying $\mathtt{M} \pi$ to the triangle
\[ P_1^* \tol[30]{p^*} P_0^* \to[30] Y \to[30] P_1^*[1] \]
we obtain the exact sequence
\[ P_1 \tol[30]{p} P_0 \to[30] \mathtt{M} \pi Y \to[30] \mathtt{M} \pi P_1^*[1]. \]
Since $\pi \Lambda$ is a cluster tilting object in $\mathscr{C}$ we have $\mathtt{M} \pi P_1^* [1] = 0$, and hence
\[ \mathtt{M} \pi Y = \Cok p = \widetilde{ \mathtt{M} X}. \qedhere \]
\end{proof}

\begin{corollary} \label{corollary.isogradresol}
Let $\Lambda$ as above, and $f \colon P \to Q$ in $\mathscr{C}$ with $P, Q \in \add \pi\Lambda$. Then $\Cone_{\mathscr{C}} f \in \Im \pi$ if and only if $f$ is isomorphic to a map in the image of $\pi$.
\end{corollary}

\begin{proof}
The ``if'' part is clear. Assume $\Cone_{\mathscr{C}} f \in \Im \pi$. Then $\Cone_{\mathscr{C}} f = \Cone_{\mathscr{C}} \pi p^*$, with $p^*$ as in the proof of \ref{theorem.imgr}. Since $\add \pi \Lambda$-resolutions are unique up to isomorphism, we have $f \iso \pi p^*$.
\end{proof}

\begin{corollary}
Let $\Lambda$ as above. Let $P$ and $Q$ be indecomposable projective $\Lambda$-modules. Then the following are equivalent:
\begin{enumerate}
\item $\{ \Cone_{\mathscr{C}} f \mid f \colon \pi P \to \pi Q \} \subseteq \Im \pi$, and
\item for any $i < j$ and any $f \colon \mathbb{S}_2^i P \to Q$ and $g \colon \mathbb{S}_2^j P \to Q$ with $f \neq 0$ there are $r \colon \mathbb{S}_2^{j-i} P \to P$ and $s \colon \mathbb{S}_2^{j-i} Q \to Q$ such that $g = \mathbb{S}_2^i(r) f + \mathbb{S}_2^{j-i}(f) s$.
\end{enumerate}
\end{corollary}

\begin{proof}
We denote by $\underline{P}$ and $\underline{Q}$ the graded $\widetilde{\Lambda}$-modules $\mathtt{M} \pi P$ and $\mathtt{M} \pi Q$, respectively, and use similar notation for maps.

Assume first (1), and let $f$ and $g$ as in (2). By Corollary~\ref{corollary.isogradresol} we have
\[ \begin{tikzpicture}[yscale=-1]
 \node (A1) at (0,0) {$\pi P$};
 \node (B1) at (3,0) {$\pi Q$};
 \node (A2) at (0,1.5) {$\pi P$};
 \node (B2) at (3,1.5) {$\pi Q$};
 \draw [->] (A1) -- node [above] {$\pi f+ \pi g$} (B1);
 \draw [->] (A2) -- node [above] {$\pi h$} (B2);
 \draw [->] (A1) -- node [left] {$r$} (A2);
 \draw [->] (B1) -- node [right] {$s$} (B2); 
\end{tikzpicture} \]
for some $h: \mathbb{S}_2^k P \to Q$, and $r \in \Aut(\pi P)$, $s \in \Aut(\pi Q)$. Passing to graded $\widetilde{\Lambda}$-modules and decomposing $r$ and $s$ into their homogeneous parts we obtain
\[ \begin{tikzpicture}[yscale=-1]
 \node (A1) at (0,0) {$\underline{P}$};
 \node (B1) at (3,0) {$\underline{Q}$};
 \node (A2) at (0,1.5) {$\underline{P}$};
 \node (B2) at (3,1.5) {$\underline{Q}$};
 \draw [->] (A1) -- node [above] {$\underline{f} + \underline{g}$} (B1);
 \draw [->] (A2) -- node [above] {$\underline{h}$} (B2);
 \draw [->] (A1) -- node [left] {$\sum \underline{r_i}$} (A2);
 \draw [->] (B1) -- node [right] {$\sum \underline{s_i}$} (B2); 
\end{tikzpicture} \]
with $r_i \colon \mathbb{S}_2^i P \to P$ and $s_i \colon \mathbb{S}_2^i Q \to Q$. Since $r$ and $s$ are automorphisms so are $r_0$ and $s_0$, and we may assume $r_0 = 1 = s_0$.

Since there are no morphisms in negative degree, taking the non zero morphisms of the smallest degree of the maps $(\underline{f}+\underline{g})(\sum \underline{s_i})$ and $(\sum \underline{r_i}) \underline{h}$, one gets that $k = i$ and $h = f$. Now looking at the morphism of degree $j$ one gets $g=\mathbb{S}_2^i(r_{j-i})f - \mathbb{S}_2^{j-i}(f) s_{j-i}$, which is the factorization property of part (2) of the corollary.

Now assume (2) holds, and assume there is a map $f \colon \pi P \to \pi Q$ such that $\Cone_{\mathscr{C}} f \not\in \Im \pi$. Then $f$ is not isomorphic to a map in the image of $\pi$, that is not isomorphic to a homogeneous map. Let $f' \iso f$. We may write $f' = \sum_{i \in \mathbb{Z}} f'_i$, where $f'_i$ is homogeneous of degree $i$. By assumption at least two of these homogeneous parts do not vanish, and thus we may define
\begin{align*}
d(f') = & (d(f')_1, d(f')_2), & \text{with } & d(f')_1 = \min \{i \in \mathbb{Z} \mid f'_i \neq 0\} \text{, and} \\ &&& d(f')_2 = \min \{i > d(f')_1 \mid f'_i \neq 0\}.
\end{align*}
Note that since $\Lambda$ is $\tau_2$-finite, nonzero maps $\pi P \to \pi Q$ can only exist in finitely many degrees. In particular the set
\[ D = \{ d(f') \mid f' \iso f \} \]
is finite. Ordering pairs of integers lexicographically (that is by $(a_1, a_2) \leq (b_1, b_2)$ if $a_1 < b_1$ or ($a_1 = b_1$ and $a_2 \leq b_2$)), we may assume that $d(f)$ is maximal in $D$.

By (2) we have $r \colon \mathbb{S}_2^{d(f)_2-d(f)_1} P \to P$ and $s \colon \mathbb{S}_2^{d(f)_2-d(f)_1} Q \to Q$ such that $\pi f_j = (\pi r) (\pi f_i) + (\pi f_i) (\pi s)$. Now $f \iso (1 - \pi r) f (1 - \pi s)$. Looking at this degree wise we have
\begin{align*}
f & \iso (1 - \pi r) ( \pi f_{d(f)_1} + \pi f_{d(f)_2} + \sum_{i > d(f)_2} \pi f_i ) (1 - \pi s) \\
& = \pi f_{d(f)_1} + \pi f_{d(f)_2} - (\pi r) (\pi f_{d(f)_1}) - (\pi f_{d(f)_1}) (\pi s) \\
& \qquad \qquad + \text{ things of degree } > d(f)_2 \\
& = \pi f_{d(f)_1} + \text{ things of degree } > d(f)_2
\end{align*}
contradicting our assumption that $d(f)$ is maximal in $D$. Hence (1) must hold.
\end{proof}

\section{Gradable modules and parameter families}

In this section we show that a module over a graded algebra is either gradable, or belongs to a one-parameter family of modules which are almost all non-isomorphic. In particular it will follow that modules which represent an open orbit in the representation variety are always gradable.

\begin{definition}
Let $R$ be a $\mathbb{Z}$-graded $k$-algebra. For $\alpha \in k^{\times}$ we denote by $\sigma_{\alpha}$ the algebra-automorphism given on homogeneous elements by
\[ r \mapsto \alpha^{\deg r} r. \]
For a (non-graded) $R$-module $M$ we denote by $M_{\alpha}$ the module twisted by the automorphism $\sigma_{\alpha}$. That is, $M_{\alpha} = M$ as $k$-vector spaces, but with the new module multiplication given by
\[ m \cdot_{\alpha} r = \alpha^{\deg r} mr \]
for homogeneous $r \in R$.
\end{definition}

\begin{theorem} \label{theorem.param-ungr}
Let $R$ be a finitely generated $\mathbb{Z}$-graded $k$-algebra, and $M$ a finite-dimensional $R$-module. Then exactly one of the following happens:
\begin{enumerate}
\item The modules $M_{\alpha}$ with $\alpha \in k^{\times}$ are all isomorphic, and $M$ is gradable.
\item For any $\alpha \in k^{\times}$ there are only finitely many $\beta \in k^{\times}$ such that $M_{\alpha} \iso M_{\beta}$, and $M$ is not gradable.
\end{enumerate}
\end{theorem}

\begin{proof}
By assumption $R$ is generated (as $k$-algebra) by a finite set $\{ r_i \}$ of homogeneous elements. We set $g = \max \{ | \deg r_i | \}$.

Assume $M_{\alpha} \iso M_{\beta}$, and $(\frac{\alpha}{\beta})^i \neq 1$ for $i \leq g \cdot \dim M$. Let $\psi \colon M_{\alpha} \to M_{\beta}$ be an isomorphism. Then $\psi$ can be considered as automorphism of the $k$-vector space $M$. Let 
\[ M = \bigoplus_{\lambda \in k^{\times}} M(\lambda) \]
be the generalized eigenspace decomposition of $M$ with respect to $\psi$. Let $r \in R$ be homogeneous. For $m \in M$ we have
\begin{align*}
\left( \psi - \left( \frac{\beta}{\alpha} \right)^{\deg r} \lambda \right) (m r) & = \psi(mr) - \left( \frac{\beta}{\alpha} \right)^{\deg r} \lambda m \cdot r \\
& = \left( \frac{1}{\alpha} \right)^{\deg r} \psi(m \cdot_{\alpha} r) - \left( \frac{\beta}{\alpha} \right)^{\deg r} \lambda m r \\
& = \left( \frac{1}{\alpha} \right)^{\deg r} \psi(m) \cdot_{\beta} r - \left( \frac{\beta}{\alpha} \right)^{\deg r} \lambda m r \\
& = \left( \frac{\beta}{\alpha} \right)^{\deg r} \left( \psi - \lambda  \right)(m) r
\end{align*}
It follows that for $m \in M(\lambda)$ we have $mr \in M(\left( \frac{\beta}{\alpha} \right)^{\deg r} \lambda)$. If we denote by $\left< \frac{\beta}{\alpha} \right>$ the cyclic subgroup of $k^{\times}$ generated by $\frac{\beta}{\alpha}$ we obtain a direct sum decomposition of $M$ into the summands
\[ \bigoplus_{\sigma \in \left< \frac{\beta}{\alpha} \right>} M(\sigma \lambda) \]
where $\lambda$ runs over representatives of the cosets of $\left< \frac{\beta}{\alpha} \right>$ in $k^{\times}$. Moreover these summands are $\left< \frac{\beta}{\alpha} \right>$-gradable. Since the order of $\left< \frac{\beta}{\alpha} \right>$ is bigger than $g \cdot \dim M$ these $\left< \frac{\beta}{\alpha} \right>$-gradings can be lifted to $\mathbb{Z}$-gradings. Thus $M$ is gradable.

Conversely, if $M$ is gradable then it is immediate that
\begin{align*}
M & \to M_{\alpha} \\
m & \mapsto \alpha^{\deg m} m
\end{align*}
gives an isomorphism of $R$-modules for any $\alpha \in k^{\times}$.
\end{proof}

\begin{definition}
Let $R$ be a finitely generated $k$-algebra and $M$ a finite-dimensional $R$-module. We say that $M$ \emph{has an open orbit} if the orbit of $M$ (under the natural ${\rm GL}(\dim M, k)$-action) in the variety of $(\dim M)$-dimensional $R$-modules is open.
\end{definition}

With this definition we have the following immediate consequence of Theorem~\ref{theorem.param-ungr}.

\begin{corollary} \label{corollary.open-image}
Let $R$ be a finitely generated $\mathbb{Z}$-graded $k$-algebra, and $M$ a finite-dimensional $R$-module, which has an open orbit. Then $M$ is gradable.
\end{corollary}

\begin{proof}
The $M_{\alpha}$ form a line in the representation variety. Since $M$ has an open orbit they are almost all isomorphic to $M$. Hence we are in the first case of Theorem~\ref{theorem.param-ungr}.
\end{proof}

\begin{corollary} \label{corollary.rigid-image}
Let $R$ be a finitely generated $\mathbb{Z}$-graded $k$-algebra, and $M$ a finite-dimensional $R$-module, such that $\Ext^1_R(M,M) = 0$. Then $M$ is gradable.
\end{corollary}

\begin{proof}
By \cite[\S~3.5]{Voigt}, $\Ext^1_R(M,M) = 0$ implies that $M$ has an open orbit. Hence the claim follows from Corollary~\ref{corollary.open-image}.
\end{proof}

\section{Application of results on gradable modules}

In this section we use results on gradable modules by Gordon and Green \cite{GG} and from the previous section to the situation of the graded algebra $\widetilde{\Lambda}$, for a $\tau_2$-finite algebra $\Lambda$ (see Subsection~\ref{subsect.intro_Amiotcluster}). By Theorem~\ref{theorem.imgr} this yields results on the image of the derived category in the cluster category.

\subsection{Auslander-Reiten components}

We recall the following result of Gordon and Green.

\begin{theorem}[{\cite[Theorem~4.2]{GG}}] \label{theorem.gordon-green}
Let $R$ be graded finite-dimensional $k$-algebra. A component of the Auslander-Reiten quiver of (ungraded) $R$-modules either only contains gradable modules, or does not contain any gradable modules.
\end{theorem}

Applying this to our setup we obtain the following:

\begin{theorem} \label{theorem.image_ARclosed}
Let $\Lambda$ be a $\tau_2$-finite algebra. Then the image of the derived category in the cluster category is a union of Auslander-Reiten components.
\end{theorem}

For the proof we need the following observation.

\begin{observation} \label{obs.AR_decomp}
Let $\mathscr{A}$ be an Auslander-Reiten component of $\mathscr{C}_{\Lambda}$. If $\mathscr{A}$ does not contain any summand of $\pi \Lambda$, then $\mathtt{M} \mathscr{A}$ is an Auslander-Reiten component of $\mod \widetilde{\Lambda}$. If $\mathscr{A}$ contains a summand of $\pi(\Lambda)$ then $\mathtt{M} \mathscr{A}$ is a union of Auslander-Reiten components of $\mod \widetilde{\Lambda}$, and all these components contain a projective or an injective $\widetilde{\Lambda}$-module.
\end{observation}

\begin{proof}[Proof of Theorem~\ref{theorem.image_ARclosed}]
By Observation~\ref{obs.AR_decomp} we have to consider two cases:

If $\mathscr{A}$ is an Auslander-Reiten component of $\mathscr{C}_{\Lambda}$ not containing any summand of $\pi \Lambda$ then $\mathtt{M} \mathscr{A}$ is an Auslander-Reiten component of $\mod \widetilde{\Lambda}$, and by Theorems~\ref{theorem.imgr} and \ref{theorem.gordon-green} either all or no objects in $\mathscr{A}$ lie in the image of $\pi$.


Now consider a component $\mathscr{A}$ of the Auslander-Reiten quiver of $\mathscr{C}_{\Lambda}$ which contains a summand of $\pi \Lambda$. 
By Observation~\ref{obs.AR_decomp} each of the Auslander-Reiten components of $\mathtt{M} \mathscr{A}$ contains at least one projective or injective, hence gradable, module. So, by Theorem~\ref{theorem.gordon-green}, all objects in these components are gradable. Hence, by Theorem~\ref{theorem.imgr}, all objects in $\mathscr{A}$ lie in the image of $\pi$.
\end{proof}

\subsection{Rigid objects}

We call an object in a triangulated or exact category \emph{rigid} if $\Ext^1(X,X) = 0$.

\begin{theorem}
Let $\Lambda$ be a $\tau_2$-finite algebra. Let $X \in \mathscr{C}_{\Lambda}$ be rigid. Then $X$ is in the image of $\pi \colon \mathscr{D}_{\Lambda} \to \mathscr{C}_{\Lambda}$.
\end{theorem}

\begin{proof}
By \cite[Theorem~4.9]{KoenigZhu}, if $X$ is rigid then so is $\mathtt{M} X$.
By Corollary~\ref{corollary.rigid-image} any rigid module in $\mod \widetilde{\Lambda}$ is gradable, so in particular $\mathtt{M} X$ is gradable. Hence, by Theorem~\ref{theorem.imgr}, $X$ lies in the image of $\pi$.
\end{proof}

\begin{corollary}
Any cluster tilting object in $\mathscr{C}_{\Lambda}$ is the image of some object in $\mathscr{D}_{\Lambda}$.
\end{corollary}

\section{Fractional Calabi-Yau type situations}

In this section, we study the situation that there is an object in $\mathscr{D}_{\Lambda}$ satisfying a fractional Calabi-Yau type condition, that is an $X$ such that $X[a] \iso \mathbb{S}^b X$ for certain $a, b \in \mathbb{Z}$. We then show the following:

\begin{theorem} \label{theorem.fracCY}
Let $\Lambda$ be a connected $\tau_2$-finite algebra. Assume that there is some indecomposable object $X \in \mathscr{D}_{\Lambda}$ such that $X[a] \iso \mathbb{S}^b X$ for some $a, b \in \mathbb{Z}$ with $a \neq b$.

Then the functor $\pi \colon \mathscr{D}_{\Lambda} \to \mathscr{C}_{\Lambda}$ is dense if and only if $\Lambda$ is piecewise hereditary.
\end{theorem}

\begin{remark}
\begin{itemize}
\item If there is a functorial isomorphism $X[a] \iso \mathbb{S}^b X$ for all $X \in D^b(\mod \Lambda)$ then $\Lambda$ is called \emph{fractionally Calabi-Yau} of dimension $\frac{a}{b}$.
\item The condition $X[a] \iso \mathbb{S}^b X$ means that in the cluster category we have $\tau^a \pi X \iso \pi X[a] \iso \mathbb{S}^b \pi X \iso \tau^{2b} \pi X$.
\item If $\Lambda$ is piecewise hereditary, then there can only be indecomposable objects $X$ with $X[a] \iso \mathbb{S}^b X$ for $a \leq b$ (and for $a < b$ the algebra $\Lambda$ is of Dynkin type, while for $a = b$ it is of Euclidean or tubular type).
\item For algebras which are not piecewise hereditary there may be indecomposable objects $X$ satisfying $X[a] \iso \mathbb{S}^b X$ for $a > b$ or $a < b$. For instance for the algebra given by the quiver
\[ \begin{tikzpicture}
 \node (A) at (0,0) {1};
 \node (B) at (1,1) {2};
 \node (C) at (2,0) {3};
 \draw [->] (A) -- node [above left=-5pt] {$\scriptstyle \alpha$} (B);
 \draw [->] (B) -- node [above right=-5pt] {$\scriptstyle \beta$} (C);
 \draw [->] (C) -- node [below] {$\scriptstyle \gamma$} (A);
\end{tikzpicture} \]
subject to the relation $\alpha \beta$. Then the projective module $P_2$ corresponding to vertex $2$ satisfies $P_2 \iso \mathbb{S} P_2$, while the simple module $S_1$ corresponding to vertex $1$ satisfies $S_1[3] \iso \mathbb{S}^2 S_1$.
\end{itemize}
\end{remark}



\subsection{The shape of the Auslander-Reiten component}

In this subsection we study the shape of the Auslander-Reiten component of $X$ in case the assumptions of Theorem~\ref{theorem.fracCY} are satisfied. We denote this component by $\mathscr{A}_X$.

We begin with two immediate observations.

\begin{observation}[{\cite[Struktursatz~1.5]{Riedtmann}}]
The Auslander-Reiten component $\mathscr{A}_X$ is of the form $\mathbb{Z}Q / G$ for some quiver $Q$ and a group $G$ of automorphisms of $\mathbb{Z}Q$. (Actually this is true for any Auslander-Reiten component.)
\end{observation}

\begin{observation} \label{obs.cy_on_comp}
For any $Y \in \mathscr{A}_X$ we have $Y[a] \iso \mathbb{S}^b Y$.
\end{observation}

\begin{proposition} \label{prop.AR_shapes}
The Auslander-Reiten component $\mathscr{A}_X$ is of the form $\mathbb{Z}Q$, with $Q$ Dynkin or of type $A_{\infty}$.
\end{proposition}

For the proof we need the following observation.

\begin{lemma} \label{lemma.mesh_finite}
Let $Q$ be a connected quiver. In the mesh category of $\mathbb{Z} Q$ we have
\[ \Hom(Y, \tau^{-i} Y) = 0 \quad \forall Y \; \forall i \gg 0 \]
if and only if $Q$ is either a Dynkin quiver or of type $A_{\infty}$.
\end{lemma}

\begin{proof}
Assume the mesh category satisfies $ \Hom(Y, \tau^{-i} Y) = 0$ for any $Y$ and $i$ sufficiently large. Then $Q$ cannot contain any finite subquiver which is not of Dynkin type. Thus $Q$ is either a Dynkin quiver itself, or of one of the types $A_{\infty}$, $A_{\infty}^{\infty}$, or $D_{\infty}$. A straight forward calculation shows that in the latter two cases $\Hom(Y, \tau^{-i} Y) \neq 0$ for arbitrarily large $i$.

Conversely it is easy to see that for $Q$ Dynkin or of type $A_{\infty}$ we have the vanishing property of the lemma.
\end{proof}

\begin{proof}[Proof of Proposition~\ref{prop.AR_shapes}]
By Observation~\ref{obs.cy_on_comp} we have that $\mathscr{A}_X$ is of the form $\mathbb{Z}Q / G$ for some quiver $Q$ and some group $G$ of automorphisms of $\mathbb{Z}Q$.

Note that the assumption $Y[a] \iso \mathbb{S}^b Y$ can be reformulated as $\tau^b Y \iso Y[a-b]$. Recall that we also assumed $a-b \neq 0$. Hence, since $\gld \Lambda < \infty$, for any $Y \in \mathscr{A}_X$ we have $\Hom_{\mathscr{D}_{\Lambda}}(Y, \tau^{-i} Y) = 0$ for $i \gg 0$.

It follows that the same is true in the mesh category of $\mathbb{Z}Q$, and hence, by Lemma~\ref{lemma.mesh_finite}, that $Q$ is Dynkin or of type $A_{\infty}$.

Finally note that $\Hom_{\mathscr{D}_{\Lambda}}(Y, \tau^{-i} Y) = 0$ for $i \gg 0$ also implies that $G = 1$.
\end{proof}

\begin{lemma} \label{lemma.non-Dynkin}
If $\mathscr{A}_X = \mathbb{Z}Q$ with $Q$ Dynkin, then $\Lambda$ is piecewise hereditary of type $Q$.
\end{lemma}

\begin{proof}
One can calculate the dimension of the morphisms between objects in $\mathscr{A}_X$, up to morphisms in the infinite radical, by using mesh relations. Thus these dimensions coincide with the dimensions of morphism spaces in $\mathscr{D}_{kQ}$.

It follows that, for any given object of $\mathscr{A}_X$, there are only finitely many objects such that there are morphisms between the two outside the infinite radical. It follows that some power of the radical, and thus the infinite radical, vanishes. Hence one sees that there are no morphisms from $\mathscr{A}_X$ to any other component or vice versa. So, since $\Lambda$ is connected, $\mathscr{A}_X$ is the entire Auslander-Reiten quiver of $\mathscr{D}_{\Lambda}$. It follows that $\mathscr{D}_{\Lambda} = \mathscr{D}_{kQ}$ (either note that any complete slice is a tilting object, or use \cite[Theorem~7.1]{Amiot07}).
\end{proof}

\subsection{A limit of the Auslander-Reiten component}

In this subsection we focus on the case that the Auslander-Reiten component $\mathscr{A}_X$ is of type $A_{\infty}$. We label the indecomposable objects in $\mathscr{A}_X$ by $\leftsub{i}{X}_j$ with $i \leq j$ as indicated in the following picture.
\[ \begin{tikzpicture}[xscale=2.2]
 \foreach \i in {-2,...,2}
  \node (X\i\i) at (\i, 0) {$\leftsub{\i}{X}_{\i}$};
 \foreach \i/\j in {-2/-1,-1/0,0/1,1/2}
  \node (X\i\j) at (\i+.5, 1) {$\leftsub{\i}{X}_{\j}$};
 \foreach \i/\j in {-3/-1,-2/0,-1/1,0/2,1/3}
  \node (X\i\j) at (\i+1,2) {$\leftsub{\i}{X}_{\j}$};
 \foreach \i/\j in {-2/-1,-1/0,0/1,1/2}
  {
   \draw [->] (X\i\i) -- (X\i\j);
   \draw [->] (X\i\j) -- (X\j\j);
   \draw [->] (X\j\j) -- node [pos=.2] {$|$} (X\i\i);
  }
 \foreach \i/\j/\k in {-2/-1/0,-1/0/1,0/1/2}
   \draw [->] (X\j\k) -- node [pos=.2] {$|$} (X\i\j);
 \foreach \i/\j/\k/\l in {-3/-2/-1/0,-2/-1/0/1,-1/0/1/2,0/1/2/3}
  {
   \draw [->] (X\i\k) -- (X\j\k);
   \draw [->] (X\j\k) -- (X\j\l);
   \draw [->] (X\j\l) -- node [pos=.2] {$|$} (X\i\k);
  }
 \node at (-2.5,0) {$\cdots$};
 \node at (-2.2,1) {$\cdots$};
 \node at (-2.5,2) {$\cdots$};
 \node at (2.5,0) {$\cdots$};
 \node at (2.2,1) {$\cdots$};
 \node at (2.5,2) {$\cdots$};
 \foreach \i in {-2,...,2}
  \node at (\i,2.8) {$\vdots$};
\end{tikzpicture} \]
In particular $\mathbb{S}^{\ell} \leftsub{i}{X}_j = \leftsub{i-\ell}{X}_{j-\ell}[\ell]$.

\begin{proposition} \label{prop.limit}
Let $X \in \mathscr{D}_{\Lambda}$ such that $X[a] = \mathbb{S}^bX$ for some $a \neq b$, and such that the Auslander-Reiten component $\mathscr{A}_X$ of $X$ is of type $\mathbb{Z}A_{\infty}$. With the labels as above we have the following.
\begin{enumerate}
\item For any $T \in \mathscr{D}_{\Lambda}$ there are $i_0$ and $j_0$ such that
\[ \Hom_{\mathscr{D}_{\Lambda}}(T, \leftsub{i}{X}_j) \iso \Hom_{\mathscr{D}_{\Lambda}}(T, \leftsub{i_0}{X}_{j_0}) \qquad \forall i \leq i_0, j \geq j_0. \]
\item For any $T \in \mathscr{D}_{\Lambda}$ there are $i_0$ and $j_0$ such that
\[ \Hom_{\mathscr{D}_{\Lambda}}(\leftsub{i}{X}_j, T) \iso \Hom_{\mathscr{D}_{\Lambda}}(\leftsub{i_0}{X}_{j_0}, T) \qquad \forall i \leq i_0, j \geq j_0. \]
\end{enumerate}
\end{proposition}

\begin{proof}
We only prove (1), (2) is dual.

By Observation~\ref{obs.cy_on_comp} we have $[a] = \mathbb{S}^b$ on objects of the component $\mathscr{A}_X$. Hence, for $i$ and $\ell \in \mathbb{Z}$, we have
\[ \leftsub{i + \ell b}{X}_{i + \ell b} = \mathbb{S}^{- \ell b} \leftsub{i}{X}_i [\ell b] = \leftsub{i}{X}_i [\ell (b - a)]. \]
In particular, for given $i$ the space $\Hom_{\mathscr{D}_{\Lambda}}(T, \leftsub{i + \ell b}{X}_{i + \ell b})$ is non-zero for only finitely many $\ell$. It follows that $\Hom_{\mathscr{D}_{\Lambda}}(T, \leftsub{i}{X}_i) \neq 0$ for only finitely many $i$.

Thus we can choose $i_0$ and $j_0$ such that
\[ \Hom_{\mathscr{D}_{\Lambda}}(T[-1] \oplus T \oplus T[1], \leftsub{\ell}{X}_{\ell}) = 0 \qquad \forall \ell \not \in \{i_0, \ldots, j_0\}. \]
To complete the proof, it suffices to show that for $i \leq i_0$ and $j \geq j_0$ the maps $\leftsub{i-1}{X}_j \to \leftsub{i}{X}_j$ and $\leftsub{i}{X}_j \to \leftsub{i}{X}_{j+1}$ induce isomorphisms
\begin{align*}
\Hom_{\mathscr{D}_{\Lambda}}(T, \leftsub{i-1}{X}_j) & \to[30] \Hom_{\mathscr{D}_{\Lambda}}(T, \leftsub{i}{X}_j) \text{ and} \\
\Hom_{\mathscr{D}_{\Lambda}}(T, \leftsub{i}{X}_j) & \to[30] \Hom_{\mathscr{D}_{\Lambda}}(T, \leftsub{i}{X}_{j+1}) \text{, respectively.}
\end{align*}
This follows from our choice of $i_0$ and $j_0$, and the fact that the cones of the two maps are $\leftsub{i-1}{X}_{i-1}[1]$ and $\leftsub{j+1}{X}_{j+1}$, respectively.
\end{proof}

\begin{construction} \label{const.limit}
In the setup of Proposition~\ref{prop.limit}, let $\mathtt{M} \pi \mathscr{A}_X$ be the graded $\widetilde{\Lambda}$-module given by
\[ (\mathtt{M} \pi \mathscr{A}_X)_{\ell} = \Hom_{\mathscr{D}_{\Lambda}}(\mathbb{S}_2^{\ell} \Lambda, \leftsub{i}{X}_j) \quad \text{ for } i \ll 0, j \gg 0. \]
(This is well-defined by Proposition~\ref{prop.limit}.) Note that $\mathtt{M} \pi \mathscr{A}_X$ does not have finite-dimension, but its graded pieces do.
\end{construction}

\begin{lemma} \label{lemma.no_finite_summands}
The graded $\widetilde{\Lambda}$-module $\mathtt{M} \pi \mathscr{A}_X$ does not have any non-zero finite-dimensional direct summands.
\end{lemma}

In the proof we will use the following observation:

\begin{observation} \label{obs.hom_bounded_unbounded}
Let $\Gamma$ be a $\mathbb{Z}$-graded algebra, which is generated over $\Gamma_0$ by $\Gamma_1$. For a graded $\Gamma$-module $M$, and a finite interval $I \subset \mathbb{Z}$, we denote by $M_I$ the module coinciding with $M$ in degrees in $I$, and vanishing in degrees outside $I$. Then
\begin{enumerate}
\item If $M \in \mod\gr \Gamma$ is concentrated in degrees $\ell_{\min}, \ldots, \ell_{\max}$, and $N \in \mod\gr \Gamma$, then 
\begin{align*}
\Hom_{\mod\gr \Gamma}(M, N) & = \{ (f_i)_{i \in \mathbb{Z}} \in \prod_{i \in \mathbb{Z}} \Hom_{\Gamma_0}(M_i, N_i) \mid \forall g \in \Gamma_1 \\ & \qquad \qquad \forall m \in M_i \colon gf_i(m) = f_{i+1}(gm) \} \\
& = \Hom_{\mod\gr \Gamma}(M, N_{[\ell_{\min}, \ell_{\max} + 1]}),
\end{align*}
where the latter equality holds because the conditions are empty unless $M_i \neq 0$.
\item Similarly, for $N \in \mod\gr \Gamma$ concentrated in degrees $\ell_{\min}, \ldots, \ell_{\max}$ and  $M \in \mod\gr \Gamma$ arbitrary we have
\[ \Hom_{\mod\gr \Gamma}(M, N) = \Hom_{\mod\gr \Gamma}(M_{[\ell_{\min} - 1, \ell_{\max}]}, N). \]
\end{enumerate}
\end{observation}

\begin{proof}[Proof of Lemma~\ref{lemma.no_finite_summands}]
Assume $H$ is a finite-dimensional direct summand of $\mathtt{M} \pi \mathscr{A}_X$. Then $H$ is concentrated in finitely many degrees, say $\ell_{\min}$ to $\ell_{\max}$. By Proposition~\ref{prop.limit} and the construction of $\mathtt{M} \pi \mathscr{A}_X$ there are $i_0$ and $j_0$ such that
\[ (\mathtt{M} \pi \leftsub{i}{X}_j)_{\ell} = (\mathtt{M} \pi \mathscr{A}_X)_{\ell} \qquad \forall i \leq i_0, j \geq j_0, \ell \in \{\ell_{\min}-1, \ldots, \ell_{\max}+1\}. \]
Since $\widetilde{\Lambda}$ is generated in degrees $0$ and $1$ it follows from Observation~\ref{obs.hom_bounded_unbounded} above that
\begin{align*}
& \Hom_{\mod\gr \widetilde{\Lambda}}(H, \mathtt{M} \pi \mathscr{A}_X) = \Hom_{\mod\gr \widetilde{\Lambda}}(H, \mathtt{M} \pi \leftsub{i}{X}_j) \qquad \forall i \leq i_0, j \geq j_0 \text{, and} \\
& \Hom_{\mod\gr \widetilde{\Lambda}}(\mathtt{M} \pi \mathscr{A}_X, H) = \Hom_{\mod\gr \widetilde{\Lambda}}(\mathtt{M} \pi \leftsub{i}{X}_j, H) \qquad \forall i \leq i_0, j \geq j_0.
\end{align*}
It follows that $H$ is a direct summand of $\mathtt{M} \pi \leftsub{i}{X}_j$ for any $i \leq i_0$ and $j \geq j_0$. However the graded $\widetilde{\Lambda}$-modules $\mathtt{M} \pi \leftsub{i}{X}_j$ are indecomposable (or zero), and pairwise non-isomorphic. Hence $H = 0$.
\end{proof}

\begin{lemma} \label{lemma.periodic}
We have $(\mathtt{M} \pi \mathscr{A}_X) \left< a-b \right> \iso \mathtt{M} \pi \mathscr{A}_X$.
\end{lemma}

\begin{proof}
We have
\begin{align*}
(\mathtt{M} \pi (\leftsub{i}{X}_j)) \left< a-b \right> & = \bigoplus_{\ell \in \mathbb{Z}} \Hom_{\mathscr{D}_{\Lambda}}(\mathbb{S}_2^{\ell+a-b} \Lambda, \leftsub{i}{X}_j) \\
& = \bigoplus_{\ell \in \mathbb{Z}} \Hom_{\mathscr{D}_{\Lambda}}(\mathbb{S}_2^{\ell} \Lambda, \mathbb{S}^{b-a} \leftsub{i}{X}_j [2a - 2b]) \\
& = \bigoplus_{\ell \in \mathbb{Z}} \Hom_{\mathscr{D}_{\Lambda}}(\mathbb{S}_2^{\ell} \Lambda, \mathbb{S}^{2b-a} \leftsub{i}{X}_j [a - 2b]) \\
& = \bigoplus_{\ell \in \mathbb{Z}} \Hom_{\mathscr{D}_{\Lambda}}(\mathbb{S}_2^{\ell} \Lambda, \leftsub{i-2b+a}{X}_{j-2b+a}) \\
& = \mathtt{M} \pi (\leftsub{i-2b+a}{X}_{j-2b+a}).
\end{align*}
Now the claim follows from the construction of $\mathtt{M} \pi \mathscr{A}_X$ (Construction~\ref{const.limit}).
\end{proof}

We now only have to recall the following result of Dowbor and Skowro{\'n}ski before we can give a proof for the main result of this section.

\begin{theorem}[\cite{DowborSkowronski}] \label{theo.DowborSkowronski}
Let $R$ be a finite-dimensional graded algebra. Assume there is a graded $R$-module which has finite-dimensional graded pieces, and does not have any finite-dimensional direct summands. Then there is a finite-dimensional $R$-module which is not gradable.
\end{theorem}

\begin{proof}[Proof of Theorem~\ref{theorem.fracCY}]
If $\Lambda$ is piecewise hereditary then the functor $\pi \colon \mathscr{D}_{\Lambda} \to \mathscr{C}_{\Lambda}$ is dense (see \cite[Theorem in Section~4]{K_orbit} for a formal proof).

Assume conversely that $\Lambda$ is not piecewise hereditary. By Proposition~\ref{prop.AR_shapes} and Lemma~\ref{lemma.non-Dynkin} we know that the Auslander-Reiten component $\mathscr{A}_X$ of $X$ is of type $\mathbb{Z} A_{\infty}$.

By Lemma~\ref{lemma.no_finite_summands} the graded $\widetilde{\Lambda}$-module $\mathtt{M} \pi \mathscr{A}_X$ has no finite-dimensional summands. By Lemma~\ref{lemma.periodic} it is periodic. It follows from Theorem~\ref{theo.DowborSkowronski} above that the push-down functor $\mod\gr \widetilde{\Lambda} \to \mod \widetilde{\Lambda}$ is not dense. Now the claim follows from Theorem~\ref{theorem.imgr}.
\end{proof}

\begin{example}
Let $\Lambda$ be the algebra given by the quiver
\[ \begin{tikzpicture}
 \node (A1) at (0,1) {$1$};
 \node (A2) at (1,1) {$2$};
 \node (A3) at (2,1) {$\cdots$};
 \node (A4) at (3,1) {$n$};
 \node (B1) at (0,0) {$\widetilde{1}$};
 \node (B2) at (1,0) {$\widetilde{2}$};
 \node (B3) at (2,0) {$\cdots$};
 \node (B4) at (3,0) {$\widetilde{n}$};
 \draw [->] (A1) -- (B1);
 \draw [->] (A2) -- (B2);
 \draw [->] (A4) -- (B4);
 \draw [->] (A1) -- (A2);
 \draw [->] (A2) -- (A3);
 \draw [->] (A3) -- (A4);
 \draw [->] (B1) -- (B2);
 \draw [->] (B2) -- (B3);
 \draw [->] (B3) -- (B4);
\end{tikzpicture} \]
with relations making all small squares commutative. Then the following are equivalent.
\begin{enumerate}
\item $\Lambda$ is piecewise hereditary,
\item The functor $\mathscr{D}_{\Lambda} \to \mathscr{C}_{\Lambda}$ is dense, and
\item $n \leq 5$.
\end{enumerate}
\end{example}

\begin{proof}
(1) \iff{} (3): By \cite{Sefi} the algebra $\Lambda$ is equivalent to $kA_{2n} / \Rad^3 kA_{2n}$, where $A_{2n}$ denotes a linearly oriented quiver of type $A_{2n}$. By \cite{HappelSeidel} these algebras are piecewise hereditary if and only if $2n \leq 11$, that is $n \leq 5$.

(1) \iff{} (2): Note that the algebra is isomorphic to the tensor product $kA_2 \otimes kA_n$. The Serre functor acts diagonally on modules which are tensor products of modules for $A_2$ and $A_n$. Thus
\[ \mathbb{S}^{3n+3} S_1 = \mathbb{S}^{3n+3} (S_1^{A_2} \otimes S_1^{A_n}) = \mathbb{S}^{3n+3} S_1^{A_2} \otimes \mathbb{S}^{3n+3} S_1^{A_n}, \]
where $S_1^{A_2}$ and $S_1^{A_n}$ denote the simple projective modules for $kA_2$ and $kA_n$, respectively. Since $\mathbb{S}^3 S_1^{A_2} = S_1[1]$ and $\mathbb{S}^{n+1} S_1^{A_n} = [n-1]$ we obtain
\[ \mathbb{S}^{3n+3} S_1 = S_1^{A_2}[n+1] \otimes S_1^{A_n}[3n-3] = S_1[4n-2]. \]
Now note that for $n \neq 5$ we have $4n-2 \neq 3n+3$. Thus, by Theorem~\ref{theorem.fracCY}, for $n \neq 5$ we have the equivalence (2) \iff{} (3).

Finally note that for $n = 5$ we know that (1) holds, so (2) also holds by \cite{K_orbit}.
\end{proof}

\begin{remark}
Using the same arguments as in the proof of Theorem~\ref{theorem.fracCY}, one can show that, provided there is an indecomposable $\Lambda$-module $X$ satisfying $X[a] \iso \mathbb{S}^b X$ for some $a \neq b$, the answer to Skowro{\'n}ski's question (see \ref{question.skowronski}) is that all $T(\Lambda)$ modules are gradable precisely if $\Lambda$ is piecewise hereditary.
\end{remark}

\section{Oriented cycles}

\begin{theorem} \label{thm.cycles}
Let $\Lambda$ be a $\tau_2$-finite algebra such that the quiver of $\Lambda$ contains an oriented cycle. Then the functor $\pi \colon \mathscr{D}_{\Lambda} \to \mathscr{C}_{\Lambda}$ is not dense.
\end{theorem}

Before we give a proof for this result in Subsection~\ref{subsect.proof_cycles} we need to prepare a technical result on the indecomposability of certain zigzag-shaped complexes in Subsection~\ref{subsect.zigzag}.

\subsection{Indecomposability of zigzags} \label{subsect.zigzag}

\begin{lemma} \label{lemma.zigzag}
Let $\mathscr{A}$ be an additive $k$-category. Assume we are given the following objects and morphisms
\[ \begin{tikzpicture}[xscale=2.3,yscale=1.5]
 \foreach \i in {-1,...,2}
  \node (A\i) at (\i, 1) {$A_{\i}$};
 \foreach \i in {-2,...,2}
  \node (B\i) at (\i+.5, 0) {$B_{\i}$};
 \foreach \i/\j in {-2/-1,-1/0,0/1,1/2}
  {
   \draw [->] (A\j) -- node [above left=-4pt] {$g_{\j}$} (B\i);
   \draw [->] (A\j) -- node [above right=-4pt] {$f_{\j}$} (B\j);
  }
 \node at (-2,.5) {$\cdots$};
 \node at (3,.5) {$\cdots$};
\end{tikzpicture} \]
such that for all $i$
\begin{itemize}
\item $A_i$ and $B_i$ have local endomorphism ring,
\item $f_i \not \in \Hom_{\mathscr{A}}(A_i, A_{i+1}) \cdot g_{i+1} \cdot \End_{\mathscr{A}}(B_i) \\ \text{ } \qquad + \End_{\mathscr{A}}(A_i) \cdot g_i \cdot \Hom_{\mathscr{A}}(B_{i-1}, B_i)$, and
\item $g_i \not \in \End_{\mathscr{A}}(A_i) \cdot f_i \cdot \Hom_{\mathscr{A}}(B_{i+1}, B_i) \\ \text{} \qquad + \Hom_{\mathscr{A}}(A_i, A_{i-1}) \cdot f_{i-1} \cdot \End_{\mathscr{A}}(B_{i-1})$.
\end{itemize}
(Note that the last two requirements essentially mean that no morphism factors through its neighbors.) Assume moreover that the $A_i$ are pairwise non-isomorphic, and the $B_i$ are pairwise non-isomorphic.

Then the complex
\[ \coprod A_i \tol[140]{\left( \begin{smallmatrix} \ddots & \ddots && \\ & f_i & g_i & \\ && f_{i-1} & \ddots \\ &&& \ddots \end{smallmatrix} \right)} \coprod B_i \]
is indecomposable.
\end{lemma}

\begin{proof}
Since the $A_i$ are pairwise non-isomorphic and have local endomorphism rings any idempotent endomorphism of $\coprod A_i$ is of the form $\pi_I + (r_{ij})$ for some $I \subseteq \mathbb{Z}$, $\pi_I$ the projection to the summands $A_i$ with $i \in I$, and radical morphisms $r_{ij} \colon A_i \to A_j$. Similarly an idempotent of $\coprod B_i$ is of the form $\pi_J + (s_{ij})$.

Thus an idempotent of the complex of the lemma is a commutative diagram as below.
\[ \begin{tikzpicture}[xscale=4,yscale=2]
 \node (A) at (0,1) {$\coprod A_i$};
 \node (B) at (1,1) {$\coprod B_i$};
 \node (C) at (0,0) {$\coprod A_i$};
 \node (D) at (1,0) {$\coprod B_i$};
 \draw [->] (A) -- node [above] {$( \ddots )$} (B);
 \draw [->] (C) -- node [above] {$( \ddots )$} (D);
 \draw [->] (A) -- node [left] {$\pi_I + (r_{ij})$} (C);
 \draw [->] (B) -- node [right] {$\pi_J + (s_{ij})$} (D);
\end{tikzpicture} \]
Assume for some $i$ we have $i \in I$ but $i \not\in J$. Composing the above diagram with the injection of $A_i$ and the projection to $B_i$ we obtain the following diagram (omitting summands that do not contribute anything).
\[ \begin{tikzpicture}[xscale=4,yscale=2]
 \node (A) at (0,1) {$A_i$};
 \node (B) at (1,1) {$B_{i-1} \oplus B_i$};
 \node (C) at (0,0) {$A_i \oplus A_{i+1}$};
 \node (D) at (1,0) {$B_i$};
 \draw [->] (A) -- node [above] {$(g_i, f_i)$} (B);
 \draw [->] (C) -- node [above] {$\left( \begin{smallmatrix} f_i \\ g_{i+1} \end{smallmatrix} \right)$} (D);
 \draw [->] (A) -- node [left] {$(1 + r_{ii}, r_{i,i+1})$} (C);
 \draw [->] (B) -- node [right] {$\left( \begin{smallmatrix} s_{i-1,i} \\ s_{ii} \end{smallmatrix} \right)$} (D);
\end{tikzpicture} \]
Thus
\begin{align*}
&& (g_i, f_i) \begin{pmatrix} s_{i-1,i} \\ s_{ii} \end{pmatrix} & = (1 + r_{ii}, r_{i,i+1}) \begin{pmatrix} f_i \\ g_{i+1} \end{pmatrix} \\
& \iff & g_i s_{i-1,i} + f_i s_{ii} & = f_i + r_{ii} f_i + r_{i,i+1} g_{i+1} \\
& \iff & f_i + r_{ii} f_i - f_i s_{ii} & = g_i s_{i-1,i} - r_{i,i+1} g_{i+1}
\end{align*}
Note that $r_{ii} \otimes 1 - 1 \otimes s_{ii} \in \Rad (\End_{\mathscr{A}}(A_i) \otimes \End_{\mathscr{A}}(B_i))$, so $1 \otimes 1 + r_{ii} \otimes 1 - 1 \otimes s_{ii}$ is invertible. Thus
\[ f_i \in \Hom_{\mathscr{A}}(A_i, A_{i+1}) \cdot g_{i+1} \cdot \End_{\mathscr{A}}(B_i) + \End_{\mathscr{A}}(A_i) \cdot g_i \cdot \Hom_{\mathscr{A}}(B_{i-1}, B_i) \]
contradicting the third point of the assumption.

Similarly the cases
\begin{itemize}
\item $i \in I$ and $i-1 \not\in J$,
\item $i \not\in I$ and $i \in J$, and
\item $i \not\in I$ and $i-1 \in J$
\end{itemize}
lead to contradictions.

It follows that $I = J = \emptyset$ or $I = J = \mathbb{Z}$, so the idempotent is trivial.
\end{proof}

For the rest of this section the following piece of notation will be useful.

\begin{definition}
Let $\Lambda$ be a finite-dimensional algebra. Let $P_0$ and $P_{\ell}$ be indecomposable projective modules. A \emph{sequence of minimal relations} of length $\ell$ from $P_0$ to $P_{\ell}$ is a sequence
\[ \top P_0 = S_0, S_1, \ldots, S_{\ell} = \top P_{\ell} \]
of simple modules, such that $\Ext_{\Lambda}^2(S_i, S_{i-1}) \neq 0$ for all $i \in \{1, \ldots, \ell\}$. (Note that a $2$-extension between simples corresponds to a minimal relation in the opposite direction in the quiver of the algebra; this motivates the name.)
\end{definition}

\begin{remark} \label{rem.seq_rel}
For a $\tau_2$-finite algebra $\Lambda$ we have the following:
\begin{enumerate}
\item A minimal sequence of relations corresponds to a sequence of degree $1$ arrows in the quiver of $\widetilde{\Lambda}$, and thus to a map of degree $\ell$.
\item By Theorem~\ref{theo.tensoralg} there are no relations in $\widetilde{\Lambda}$ having summands which are products of degree $1$-arrows. Thus the map corresponding to a minimal sequence of relations is non-zero, and not equal to any other linear combination of paths.
\item In particular the length of sequences of minimal relations is bounded above by the maximal degree of $\widetilde{\Lambda}$.
\end{enumerate}
\end{remark}

\begin{proposition} \label{prop.zigzag}
Let $\Lambda$ be a $\tau_2$-finite algebra. Assume there are indecomposable projective modules $P_1, \ldots, P_{\ell}$ and $Q_1, \ldots, Q_{\ell}$ such that for all $i \in \{1, \ldots, \ell\}$ we have
\begin{itemize}
\item there is a non-zero non-isomorphisms $P_i \to Q_i$,
\item there is a sequence of minimal relations from $Q_{i-1}$ to $P_i$ (here $Q_0 = Q_{\ell}$).
\end{itemize}
Then the functor $\pi \colon \mathscr{D}_{\Lambda} \to \mathscr{C}_{\Lambda}$ is not dense.
\end{proposition}

\begin{proof}[Proof of Proposition~\ref{prop.zigzag}]
By Theorem~\ref{theorem.imgr} it suffices to show that there is a non-gradable $\widetilde{\Lambda}$-module. For this, by Theorem~\ref{theo.DowborSkowronski}, it suffices to find an indecomposable periodic graded $\widetilde{\Lambda}$-module with finite dimensional graded pieces. Finally note that a graded $\widetilde{\Lambda}$-module is indecomposable and periodic if and only if its projective presentation is.

Now apply Lemma~\ref{lemma.zigzag} for $\mathscr{A} = \proj\gr \widetilde{\Lambda}$, $A_i = P_{\overline{\iota}} \left< a_i \right>$ for $\overline{\iota} - i \in \ell \mathbb{Z}$ and certain $a_i$, and $B_i = Q_{\overline{\iota}} \left< a_i \right>$. Further we let $f_i \colon A_i \to B_i$ be a map from the first point of the proposition, and $g_i \colon A_i \to B_{i-1}$ map corresponding to the sequence of minimal relations in the second point of the proposition.

Since the $g_i$ are all of positive degree, and the $f_i$ are of degree $0$, it follows that the $a_i$ are pairwise different. It follows that the $A_i$ and the $B_i$ are pairwise non-isomorphic.

The three points of Lemma~\ref{lemma.zigzag} are easily verified: The second one holds since the $f_i$ are of degree $0$ and the $g_i$ are of positive degree. The final one follows form Remark~\ref{rem.seq_rel}(2).

Thus, by Lemma~\ref{lemma.zigzag}, the complex
\[ \coprod A_i \to[40] \coprod B_i \]
is indecomposable. It is periodic by construction, and has finite dimensional graded pieces since $\coprod B_i$ does. Thus the claim follows.
\end{proof}

\subsection{Proof of Theorem~\ref{thm.cycles}} \label{subsect.proof_cycles}

\begin{lemma} \label{lemma.relation_exists}
Let $\Lambda$ be a finite-dimensional algebra. Assume there is a sequence of arrows
\[ i_0 \tol[30]{\scriptstyle \alpha_1} i_1 \tol[30]{\scriptstyle{\alpha_2}} \cdots \tol[30]{\scriptstyle{\alpha_n}} i_n \] in the quiver of $\Lambda$, such that the corresponding map $P_{i_0} \to P_{i_n}$ vanishes. Then there are $a < b$ such that there is a minimal relation $i_a \to i_b$. Equivalently $\Ext_{\Lambda}^2(\top P_{i_b}, \top P_{i_a}) \neq 0$.
\end{lemma}

\begin{proof}
Since the map $P_{i_0} \to P_{i_n}$ vanishes we have
\[ \alpha_1 \cdots \alpha_n = \sum_i x_i r_i y_i \]
for some $x_i$ and $y_i$, and minimal relations $r_i$. With respect to the basis consisting of paths in the quiver we see that at least one $r_i$ has a non-zero scalar multiple of a subpath $\alpha_{a+1} \cdots \alpha_b$ of $\alpha_1 \cdots \alpha_n$ as a summand. Thus we have a minimal relation $i_a \to i_b$.
\end{proof}

\begin{proof}[Proof of Theorem~\ref{thm.cycles}]
We may assume that the oriented cycle in the quiver of $\Lambda$ is
\[ 0 \to 1 \to \cdots \to n \to 0. \]
Since $\Lambda$ is finite-dimensional the map corresponding to some power of this cycle vanishes, so, by Lemma~\ref{lemma.relation_exists}, there is a minimal relation between two vertices of the cycle. We choose a maximal sequence of minimal relations between vertices of the cycle; note that such a maximal sequence exists by Remark~\ref{rem.seq_rel}(3). We may assume the last relation ends in $0$, and the first one starts in $m \neq 0$ (if $m = 0$ then there is a cyclic sequence of minimal relations, hence there are arbitrarily long sequences of minimal relations, contradicting Remark~\ref{rem.seq_rel}(3)). Thus we obtain a setup as indicated in the following picture
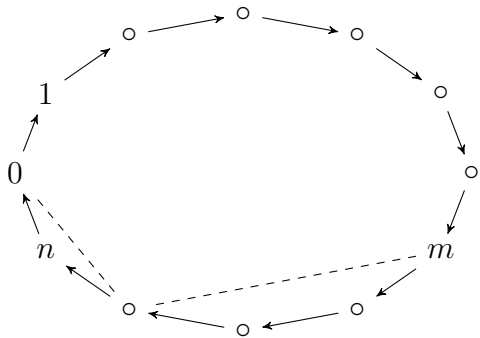
\[ \begin{tikzpicture}[yscale=.7,xscale=-1]
 \node (0) at (0:3) {$0$};
 \node (1) at (30:3) {$1$}; 
 \node (7) at (210:3) {$m$};
 \node (11) at (330:3) {$n$};
 \foreach \i in {2,3,4,5,6,8,9,10}
  \node (\i) at (30*\i:3) {$\circ$};
 \foreach \i/\j in {0/1,1/2,2/3,3/4,4/5,5/6,6/7,7/8,8/9,9/10,10/11,11/0}
  \draw [->] (\i) -- (\j);
 \draw [dashed] (7) -- (10);
 \draw [dashed] (10) -- (0);
\end{tikzpicture} \]
such that no relation among the vertices $0, \ldots, m$ starts in $0$ or ends in $m$.

If the map $P_0 \to P_m$ corresponding to the upper part of the cycle is non-zero, then we are done by Proposition~\ref{prop.zigzag} (with $\ell = 1$, the map for the first point being the one assumed to be non-zero, and the sequence of relations being the sequence from the lower part of the picture above).

Assume now the map $P_0 \to P_m$ vanishes. Then there is at least one minimal relation $i \to j$ with $0 < i < j < m$. Choose a minimal sequence of such relations, such that every one begins where the one before ends. We obtain a setup as indicated in the following picture:
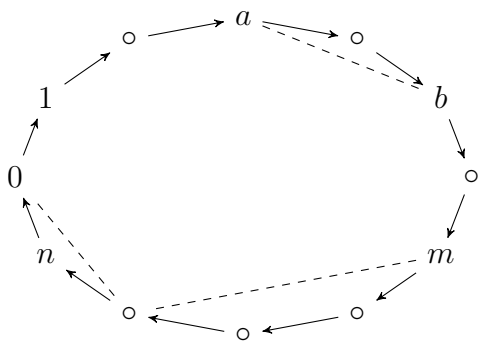
\[ \begin{tikzpicture}[yscale=.7,xscale=-1]
 \node (0) at (0:3) {$0$};
 \node (1) at (30:3) {$1$};
 \node (3) at (90:3) {$a$};
 \node (5) at (150:3) {$b$};
 \node (7) at (210:3) {$m$};
 \node (11) at (330:3) {$n$};
 \foreach \i in {2,4,6,8,9,10}
  \node (\i) at (30*\i:3) {$\circ$};
 \foreach \i/\j in {0/1,1/2,2/3,3/4,4/5,5/6,6/7,7/8,8/9,9/10,10/11,11/0}
  \draw [->] (\i) -- (\j);
 \draw [dashed] (7) -- (10);
 \draw [dashed] (10) -- (0);
 \draw [dashed] (3) -- (5);
\end{tikzpicture} \]
Assume the maps $P_0 \to P_a$ and $P_b \to P_m$ corresponding to the remaining paths (those not covered by relations) are both non-zero. Then we are done by Proposition~\ref{prop.zigzag} (with $\ell = 2$, the maps for the first point being those assumed to be non-zero, and the sequences of relations being the two sequences of relation in the picture above).

If one of the maps $P_0 \to P_a$ and $P_b \to P_m$ is non-zero iterate the argument (i.e.\ find a sequence of relations on it), until all parts of the cycle not covered by relations correspond to non-zero maps.
\end{proof}

\begin{example}
Let $n \in \mathbb{N}$, and $0 < r_1 < s_1 < r_2 < s_2 < \cdots < r_{\ell} < s_{\ell} \leq n$ with $\ell \geq 1$. Let $\Lambda$ be the algebra given by the cyclic quiver
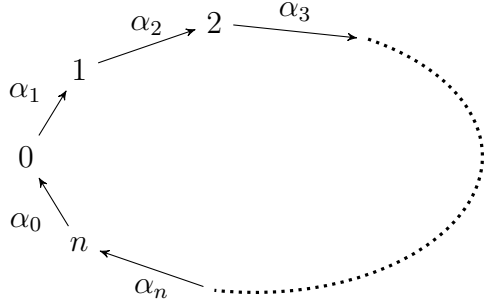
\[ \begin{tikzpicture}[yscale=.6,xscale=-1]
 \node (0) at (0:3) {$0$};
 \node (1) at (40:3) {$1$};
 \node (2) at (80:3) {$2$};
 \node (3) at (120:3) {};
 \node (7) at (280:3) {};
 \node (8) at (320:3) {$n$};
 \draw [->] (7) -- node [below] {$\alpha_n$} (8);
 \draw [->] (8) -- node [below left] {$\alpha_0$} (0);
 \draw [->] (0) -- node [above left] {$\alpha_1$} (1);
 \draw [->] (1) -- node [above] {$\alpha_2$} (2);
 \draw [->] (2) -- node [above] {$\alpha_3$} (3);
 \draw [very thick, dotted] (120:3) arc (120:280:3);
\end{tikzpicture} \]
with relations $\{ \alpha_{r_i} \cdots \alpha_{s_i} \mid i \in \{1, \ldots, \ell\}\}$.

Then $\Lambda$ is $\tau_2$-finite, and the functor $\mathscr{D}_{\Lambda} \to \mathscr{C}_{\Lambda}$ is not dense.
\end{example}

\begin{proof}
We first check that $\Lambda$ is $\tau_2$-finite. Since the relations do not overlap one easily sees that the algebra has global dimension $2$. In the quiver of $\widetilde{\Lambda}$ the arrows of degree $1$ are of the form $b_i \to a_i-1$. The only way to come back to higher labels is via the sequence of arrows $\alpha_{a_i} \cdots \alpha_{b_i}$. But this sequence is a relation. Thus $\widetilde{\Lambda}$ is finite-dimensional, and hence $\Lambda$ is $\tau_2$-finite.

The fact that the functor $\mathscr{D}_{\Lambda} \to \mathscr{C}_{\Lambda}$ is not dense now follows immediately from Theorem~\ref{thm.cycles}.
\end{proof}

\end{document}